\documentclass[12pt,a4paper]{article}
\usepackage{amsthm}
\usepackage{amsfonts}
\usepackage{mathrsfs}
\usepackage{amscd,amssymb,amsmath,graphicx,verbatim,xy,xypic,bm}
\usepackage[pdftex]{hyperref}
\usepackage[TS1,OT1,T1]{fontenc}
\usepackage{geometry}
\usepackage[multiple]{footmisc}

\newtheorem{theorem}{Theorem}[section]
\newtheorem{lemma}[theorem]{Lemma}

\newtheorem{proposition}[theorem]{Proposition}
\theoremstyle{definition}
\newtheorem{notion}[theorem]{}
\newtheorem{definition}[theorem]{Definition}
\newtheorem{example}[theorem]{Example}
\newtheorem{remark}[theorem]{Remark}

\title{Stable homotopy, 1-dimensional NCCW complexes, and Property (H)}

\makeatletter
\renewcommand\@date{{%
  \vspace{-\baselineskip}%
  \large\centering
  \begin{tabular}{@{}c@{}}
    Qingnan An
  \end{tabular},~
  \begin{tabular}{@{}c@{}}
  George A. Elliott
  \end{tabular},~
  \begin{tabular}{@{}c@{}}
  Zhichao  Liu
  \end{tabular},~~and
  \begin{tabular}{@{}c@{}}
  Yuanhang Zhang
  \end{tabular}
}}
\makeatother

\begin{document}
\maketitle

\begin{abstract}
In this paper, we show that
the homomorphisms between two unital one-dimensional NCCW complexes with the same KK-class are stably homotopic, i.e., with adding on a common homomorphism (with finite dimensional image), they are homotopic. As a consequence, any one-dimensional NCCW complex has the Property (H).


{\sl 2010 MR Subject Classification:} Primary 46L35;
Secondary 46L80, 19K14, 19K33, 19K35.
\end{abstract}
\section{Introduction}
\begin{definition}
Let $A,B$ be $\mathrm{C}^*$-algebras and $\psi_1,\psi_2$ be homomorphisms from $A$ to $B$. We say $\psi_1,\psi_2$ are \textit{stably homotopic}, if there exists a homomorphism $\eta$ from $A$ to $M_r(B)$ for some integer $r$ such that
$$\psi_1\oplus \eta\sim_h\psi_2\oplus \eta,$$
i.e., $\psi_1\oplus \eta$ and $\psi_2\oplus \eta$ are homotopic as homomorphisms from $A$ to $M_{r+1}(B)$.
In particular, stably homotopic homomorphisms induce the same $\mathrm{KK}$-class.
\end{definition}
Note that the homotopy classes of homomorphisms from $A$ to $B\otimes \mathcal{K}$ forms a abelian semigroup $[A,B\otimes\mathcal{K} ]$. Note that there is a natural map $\gamma:\,[A,B\otimes\mathcal{K} ]\to \mathrm{KK}(A,B)$, while $\gamma$ is not always injective,
in fact, from the construction of Grothendieck group of $[A,B\otimes\mathcal{K} ]$, we know that, as mentioned above, two homomorphisms, which are not homotopic with each other but stably homotopic with each other, will induce the same KK-class.
So a reasonable question is, conversely, if two homomorphisms induce the same KK-class, are they stably homotopic? Further, if they are stably homotopic but not homotopic, what is the obstruction?

We will consider these questions in the class $\mathcal{C}$ of sub-homogeneous algebras.
Making such questions clear, on one hand, will be helpful to understand the structure of the KK-group which is significant for classification theory of $\mathrm{C}^*$-algebras, and on the other hand, will lead a proof about the Property (H) (Theorem 4.8) which is quite important in the uniqueness theorem (see Lemma 7.5 of \cite{ALZ} as an application for a special case).
(Property (H) is introduced by Dadarlat in \cite{D1}, and used in the classification of $\mathrm{C}^*$-algebras, especially in uniqueness part, see \cite{D1} and \cite{DG}.)

The following is the class of algebras we are concerned with:
\begin{definition}[The class $\mathcal{C}$]\label{ET}
Let $F_1$ and $F_2$ be finite dimensional ${\mathrm C}^*$-algebras and let
$\varphi_0,\,\varphi_1:\,F_1\to F_2$ be homomorphisms.
Set
\begin{align*}
A&=A(F_1,F_2,\varphi_0,\varphi_1)
\\
&=\{(f,a)\in  C([0,1],F_2) \oplus F_1:\,f(0)=\varphi_0(a)\,  {\rm and}\, f(1)=\varphi_1(a)\}.
\end{align*}
Denote by $\mathcal{C}$ the class of all {\bf unital} such $\mathrm{C}^*$-algebras.
\end{definition}
The ${\mathrm C}^*$-algebras constructed in this way have been studied
by Elliott and Thomsen \cite{ET} (see also \cite{Ell} and \cite{Th}), which are sometimes called Elliott-Thomsen algebras or one-dimensional non-commutative finite CW complexes (NCCW).

We begin with three examples with different types of obstructions, which exist for homomorphisms but vanish in the stable sense.
\begin{example}\label{EX 1}
Let $A=A(F_1,F_2,\varphi_0,\varphi_1)$, where $F_1=\mathbb{C}\oplus\mathbb{C}\oplus\mathbb{C}$, $F_2=M_2(\mathbb{C})$, and
$$\varphi_0(a\oplus b\oplus c)=\bigg(\begin{array}{cc}
a& \\[1.5 mm]
 &c
\end{array}\bigg),
\quad
\varphi_1(a\oplus b\oplus c)=\bigg(\begin{array}{cc}
b& \\[1.5 mm]
 &c
\end{array}\bigg).
$$

Define three homomorphisms $\delta_1$, $\delta_2$, and $\delta_3:\,A\to \mathbb{C}$:
$$
\delta_1 (f,a\oplus b\oplus c)=a,\quad (f,a\oplus b\oplus c)\in A,
$$
$$
\delta_2 (f,a\oplus b\oplus c)=b,\quad(f,a\oplus b\oplus c)\in A,
$$
and
$$
\delta_3 (f,a\oplus b\oplus c)=c,\quad(f,a\oplus b\oplus c)\in A.
$$
Note that we have
${\rm Hom}(A,\mathbb{C})=\{0,\delta_1,\delta_2,\delta_3\}$ (all the homomorphisms from $A$ to $\mathbb{C}$). Suppose that
$$
\delta_1\sim_h \delta_2,
$$
and denote the homotopy path in ${\rm Hom}(A,\mathbb{C})$  by $\Phi_\theta$, $\theta\in [0,1]$. Then
$
\{\Phi_\theta(g)\mid\, \theta\in [0,1]\},
$
is a connected subset of $\mathbb{C}$, for any $g\in A$. But with
$$
g_0=(\left(
\begin{array}{cc}
  t  &  \\
   & 0
\end{array}
\right),~0\oplus 1\oplus 0)\in A,
$$
we have $\{\Phi_\theta(g_0)\mid\, \theta\in [0,1]\}=\{0,1\}$, which is a contradiction.

So
$$
\delta_1\nsim_h \delta_2,
\quad
{\rm{but}}
\quad
\delta_1\oplus\delta_3 \sim_h \delta_2\oplus\delta_3.
$$
The stable homotopy path can be chosen as $H_\theta(f, a\oplus b\oplus c)=f(\theta)$, $\theta\in [0,1]$.
In particular,
$$
\mathrm{KK}(\delta_1)=\mathrm{KK} (\delta_2).
$$
\end{example}

\begin{example}\label{EX 2}
Let $F_1=\mathbb{C}$, $F_2=\mathbb{C}\oplus\mathbb{C}$, and $\varphi_0=\varphi_1: F_1\ni a \mapsto a\oplus a\in F_2$.
Set $A=A(F_1,F_2,\varphi_0,\varphi_1)\cong C(X)$; in fact, $X$ is the topological space ``8'' and $B=\widetilde{C_0(0,1)}\cong C(S^1)$.

Construct homomorphisms $\psi_1,\psi_2$ from $A$ to $B$,
$$\psi_1(f_1\oplus f_2,a)=\begin{cases}
  f_1(4t), & \mbox{if } 0\leq t\leq \frac{1}{4} \\
  f_2(4t-1), & \mbox{if } \frac{1}{4}\leq t\leq \frac{1}{2} \\
  f_1(3-4t), & \mbox{if } \frac{1}{2}\leq t\leq \frac{3}{4} \\
  f_2(4t-3), & \mbox{if } \frac{3}{4}\leq t\leq 1
\end{cases}$$
and
$$\psi_2(f_1\oplus f_2,a)=\begin{cases}
  f_2(2t), & \mbox{if } 0\leq t\leq \frac{1}{2} \\
  f_2(2t-1), & \mbox{if } \frac{1}{2}\leq t\leq 1
\end{cases}.$$
Note that the fundamental group of $X$ is the free group $\mathbb{F}_2$ with two generators $\{x,y\}$.
By Gelfand transform, the homomorphisms $\psi_1,\psi_2$ induce two continuous maps $\psi_1',\psi_2'$ from $S^1$
to $X$ and
$$[\psi_1']=xyx^{-1}y\quad {\rm and}\quad [\psi_2']=y^2. $$
By Gelfand transform, we know that $\psi_1\nsim_h\psi_2$.
But we have $\psi_1\oplus \delta\sim_h\psi_2\oplus\delta$,
where
$$
\delta(f_1\oplus f_2,a)=a.
$$
(A concrete path can be obtained by applying the trick in Remark \ref{trick} to points $\frac{1}{4}, \frac{1}{2}$, and $\frac{3}{4}$.)
In particular, $\mathrm{KK}(\psi_1)=\mathrm{KK}(\psi_2)$.
\end{example}

\begin{example}\label{EX 3}
Set $A=M_2(\mathbb{C})$, $B=M_2(C(S^1))$, and
$$
u=\left(
\begin{array}{cc}
  z  &  \\
   & 1
\end{array}
\right)\in U(B).
$$
Consider the natural embedding map $\iota$ from $A$ to $B$, $\iota(a)=a$. Set $\psi_1=\iota$ and $\psi_2=u\cdot \iota \cdot u^*$.
We point out that
$$
\mathrm{KK}(\psi_1)=\mathrm{KK}(\psi_2)=1\in\mathbb{Z}\cong \mathrm{KK}(A,B).
$$

Note that if the winding number corresponding to a unitary $U$ in $B$ is $2n$,
then there exists a unitary path $H_t$ in $B$ with $H_0=U$ and
$$
H_1=\left(
\begin{array}{cc}
  z^n &  \\
   & z^n
\end{array}
\right)
$$
such that
$$
U \cdot \iota \cdot U^*\sim_h H_1\cdot \iota\cdot H_1^*=\iota\quad {\rm by}\quad H_t\cdot \iota\cdot H_t^*.
$$
But we cannot apply the above trick to $u$, whose winding number is 1.

In fact, every unital homomorphism from $A$ to $B$ corresponds to a continuous map from $S^1$ to ${\rm Aut}(M_2(\mathbb{C}))$, and the fundamental group of the space of homomorphisms is the same as the fundamental group of $Aut(M_2(\mathbb{C}))$.
As $$\pi_1({\rm Aut}(M_2(\mathbb{C})))=\pi_1(U(2)/\{\lambda \cdot 1_2:\,\lambda\in\mathbb{C},\,\parallel\lambda\parallel=1\})\cong \mathbb{Z}/2\mathbb{Z},$$ we have $[\psi_1]=2\mathbb{Z}$ and $[\psi_2]=1+2\mathbb{Z}$.

However, we still have $$[u\oplus u^*]=[1_B\oplus 1_B]\in U(M_2(B))/U_0(M_2(B)),$$
i.e., there exists a unitary path $V_t$ in $M_2(B)$ with $V_0=u\oplus u^*$ and $V_1=1_B\oplus 1_B$. Then
$V_t(\psi_1\oplus0)V_t^*$ gives
$$
\psi_1\oplus (u^*\cdot 0 \cdot u)=\psi_1\oplus 0\sim_h \psi_2\oplus (u^*\cdot 0 \cdot u),
$$
where $0$ is the $0$ homomorphism from $A$ to $B$.
\end{example}
In fact, Example \ref{EX 1} shows we need more points to push the spectrum; and Example \ref{EX 2} shows that we need more points to make the homotopy class of homomorphisms become coarser ($\mathrm{KK}(C(X),C(S^1))$ is a commutative group, while fundamental group of $X$ is $\mathbb{F}_2$);
Example \ref{EX 3} shows that the fact that the whole homotopy class of a unitary $u$ does not commute with a given homomorphism may form an obstruction, but it will vanish in the stable sense.
From the tips of these examples, we pass to the general the sub-homogeneous class $\mathcal{C}$ and state our main result:
\begin{theorem}\label{main result}
Let $A,B\in \mathcal{C}$ (Definition \ref{ET}), and let $\psi_1,\psi_2$ be homomorphisms from $A$ to $B$. Then
$\psi_1$ and $\psi_2$ are stably homotopic if, and only if,
$$
\mathrm{KK}(\psi_1)=\mathrm{KK}(\psi_2).
$$
{\bf In addition, the homomorphism $\eta$ we add can be always chosen to be a homomorphism with finite dimensional image.}
\end{theorem}

This paper is organized as follows.
In Section 2 below, we  collect some basic notions and tools which we will use later.
In Section 3, we show that each homomorphism between two Elliott-Thomsen algebras could be homotopic to a
homomorphism with $m$-standard form for some positive integer $m$.
In Section 4, we prove our main result, and conclude with the consequence that
every one-dimensional NCCW complex has property (H).

\section{Preliminaries}
We first list some basic notions and tools we will use later, which also help us to understand the examples more clearly.

\begin{notion}[Section 3 of \cite{GLN}]\rm
For $A=A(F_1,F_2,\varphi_0,\varphi_1)\in\mathcal{C}$ with $K_0(F_1)={\mathbb{Z}}^p$ and   $K_0(F_2)={\mathbb{Z}}^l$, consider the short exact sequence
$$
0\rightarrow SF_2 \xrightarrow{\iota} A \xrightarrow{\pi} F_1\rightarrow0,
$$
where $SF_2=C_0(0,1)\otimes F_2$ is the suspension of $F_2$, $\iota$ is the embedding map, and $\pi(f,a)=a,\,(f,a)\in A$ .
Then one has the six-term exact sequence
$$
0\to \mathrm{K}_0(A)\xrightarrow{\pi_*} \mathrm{K}_0(F_1)\xrightarrow{\partial} \mathrm{K}_0(F_2)\xrightarrow{\iota_*} \mathrm{K}_1(A)\to 0,
$$
where $\partial=\alpha-\beta$, and $\alpha$, $\beta$ are the matrices (with entries $\alpha_{ij},\beta_{ij}\in\mathbb{N},i=1,\cdots,l,j=1,\cdots,p$) corresponding to the maps $\mathrm{K}_0(\varphi_0),\,\mathrm{K}_0(\varphi_1):K_0(F_1)\rightarrow K_0(F_2)$, respectively.
Hence,
$$
\mathrm{K}_0(A)=\ker(\alpha -\beta)\subset \mathbb{Z}^p,\quad
\mathrm{K}_1(A)=\mathbb{Z}^l /{\rm Im}(\alpha - \beta),
$$
and $$\mathrm{K}_0^+ (A)=\ker(\alpha - \beta)\cap \mathrm{K}_0^+ (F_1).$$

\end{notion}
\begin{notion}\label{notation of A B}
Throughout this paper, when talking about $\mathrm{KK}(A,B)$ with $A,\,B\in\mathcal{C}$,
we shall assume the notational convention
that
$$A=A(F_1,F_2,\varphi_0,\varphi_1),\,B=B(F_1',F_2',\varphi_0 ',\varphi_1 '),$$
with $$F_1= \bigoplus_{i=1}^p M_{k_i}(\mathbb{C}),\quad
F_2= \bigoplus_{j=1}^l M_{h_j}(\mathbb{C}),$$
and
$$
F_1 '= \bigoplus_{i'=1}^{p'} M_{{k'}_{i'}}(\mathbb{C}),\quad
F_2 '= \bigoplus_{j'=1}^{l'} M_{{h'}_{j'}}(\mathbb{C}).$$
And we will use $\alpha,\,\beta,\,\alpha'$, and $\beta'$ to denote the matrices induced by
$\varphi_{0*},\,\varphi_{1*},\,\varphi_{0*}'$ and $\varphi_{1*}'$, respectively.
\end{notion}
The following Lemma comes from the Remark 2.1 and Remark 2.2 in \cite{AE}.
\begin{lemma}\label{basichom}
Let $A,B\in \mathcal{C}$ and let $\phi:\,A\to B$ be a homomorphism. Then $\phi$ is homotopic to a new homomorphism $\psi:\,A\to B$ with
$
Sp(\pi_e\circ\psi)\subset Sp(F_1),
$
where $\pi_e$ is the natural map $\pi_e:\,B\to F_2'$.
\end{lemma}

\begin{notion}[\cite{AE}]\label{cmset}
Let $A,\,B\in \mathcal{C}$.
Denote by $C(A,B)$ the set of all the commutative diagrams
  $$
\xymatrixcolsep{2pc}
\xymatrix{
{\,\,0\,\,} \ar[r]^-{}
& {\,\,\mathrm{K}_0(A)\,\,} \ar[d]_-{\lambda_{0*}} \ar[r]^-{\pi_*}
& {\,\,\mathrm{K}_0(F_1)\,\,} \ar[d]_-{\lambda_{0}} \ar[r]^-{\alpha-\beta}
& {\,\,\mathrm{K}_1(SF_2)\,\,} \ar[d]_-{\lambda_{1}} \ar[r]^-{\iota_*}
& {\,\,\mathrm{K}_1(A)\,\,} \ar[d]_-{\lambda_{1*}} \ar[r]^-{}
& {\,\,0\,\,}\\
{\,\,0\,\,} \ar[r]^-{}
& {\,\,\mathrm{K}_0(B)\,\,} \ar[r]_-{\pi_*'}
& {\,\,\mathrm{K}_0(F_1') \,\,} \ar[r]_-{\alpha'-\beta'}
& {\,\,\mathrm{K}_1(SF_2') \,\,} \ar[r]_-{\iota_*'}
& {\,\,\mathrm{K}_1(B)\,\,} \ar[r]^-{}
& {\,\,0\,\,},}
$$
and
by $M(A,B)$ the subset of $C(A,B)$ of all the commutative diagrams
$$
\xymatrixcolsep{2pc}
\xymatrix{
{\,\,0\,\,} \ar[r]^-{}
& {\,\,\mathrm{K}_0(A)\,\,} \ar[d]_-{0} \ar[r]^-{\pi_*}
& {\,\,\mathrm{K}_0(F_1)\,\,} \ar[d]_-{\mu_{0}} \ar[r]^-{\alpha-\beta}
& {\,\,\mathrm{K}_1(SF_2)\,\,} \ar[d]_-{\mu_{1}} \ar[r]^-{\iota_*}
& {\,\,\mathrm{K}_1(A)\,\,} \ar[d]_-{0} \ar[r]^-{}
& {\,\,0\,\,}\\
{\,\,0\,\,} \ar[r]^-{}
& {\,\,\mathrm{K}_0(B)\,\,} \ar[r]_-{\pi_*'}
& {\,\,\mathrm{K}_0(F_1') \,\,} \ar[r]_-{\alpha'-\beta'}
& {\,\,\mathrm{K}_1(SF_2') \,\,} \ar[r]_-{\iota_*'}
& {\,\,\mathrm{K}_1(B)\,\,} \ar[r]^-{}
& {\,\,0\,\,}}
$$
such that there exists $\mu\in {\rm Hom}(\mathrm{K}_1(SF_2), \mathrm{K}_0(F_1 '))$ satisfying $\mu_0=\mu \circ(\alpha-\beta)$, $\mu_1=(\alpha'-\beta')\circ \mu$. Since such a diagram is completely determined by $\mu$, we may denote it by $\lambda_\mu$.
\end{notion}
\begin{notion}\rm
For two commutative diagrams $\lambda_I,\,\lambda_{II}\in C(A,B)$,
$$
\xymatrixcolsep{2pc}
\xymatrix{
{\,\,0\,\,} \ar[r]^-{}
& {\,\,\mathrm{K}_0(A)\,\,} \ar[d]_-{\lambda_{I0*}} \ar[r]^-{\pi_*}
& {\,\,\mathrm{K}_0(F_1)\,\,} \ar[d]_-{\lambda_{I0}} \ar[r]^-{\alpha-\beta}
& {\,\,\mathrm{K}_1(SF_2)\,\,} \ar[d]_-{\lambda_{I1}} \ar[r]^-{\iota_*}
& {\,\,\mathrm{K}_1(A)\,\,} \ar[d]_-{\lambda_{I1*}} \ar[r]^-{}
& {\,\,0\,\,}\\
{\,\,0\,\,} \ar[r]^-{}
& {\,\,\mathrm{K}_0(B)\,\,} \ar[r]_-{\pi_*'}
& {\,\,\mathrm{K}_0(F_1') \,\,} \ar[r]_-{\alpha'-\beta'}
& {\,\,\mathrm{K}_1(SF_2') \,\,} \ar[r]_-{\iota_*'}
& {\,\,\mathrm{K}_1(B)\,\,} \ar[r]^-{}
& {\,\,0\,\,}}
$$
and
$$
\xymatrixcolsep{2pc}
\xymatrix{
{\,\,0\,\,} \ar[r]^-{}
& {\,\,\mathrm{K}_0(A)\,\,} \ar[d]_-{\lambda_{II0*}} \ar[r]^-{\pi_*}
& {\,\,\mathrm{K}_0(F_1)\,\,} \ar[d]_-{\lambda_{II0}} \ar[r]^-{\alpha-\beta}
& {\,\,\mathrm{K}_1(SF_2)\,\,} \ar[d]_-{\lambda_{II1}} \ar[r]^-{\iota_*}
& {\,\,\mathrm{K}_1(A)\,\,} \ar[d]_-{\lambda_{II1*}} \ar[r]^-{}
& {\,\,0\,\,}\\
{\,\,0\,\,} \ar[r]^-{}
& {\,\,\mathrm{K}_0(B)\,\,} \ar[r]_-{\pi_*'}
& {\,\,\mathrm{K}_0(F_1') \,\,} \ar[r]_-{\alpha'-\beta'}
& {\,\,\mathrm{K}_1(SF_2') \,\,} \ar[r]_-{\iota_*'}
& {\,\,\mathrm{K}_1(B)\,\,} \ar[r]^-{}
& {\,\,0\,\,},}
$$
define the sum of $\lambda_I$ and $\lambda_{II}$ as
$$
\xymatrixcolsep{2pc}
\xymatrix{
{\,\,0\,\,} \ar[r]^-{}
&{\,\,\mathrm{K}_0(A)\,\,} \ar[d]_-{\lambda_{I0*}+\lambda_{II0*}} \ar[r]^-{\pi_*}
& {\,\,\mathrm{K}_0(F_1)\,\,} \ar[d]_-{\lambda_{I0}+\lambda_{II0}} \ar[r]^-{\alpha-\beta}
& {\,\,\mathrm{K}_1(SF_2)\,\,} \ar[d]_-{\lambda_{I1}+\lambda_{II1}} \ar[r]^-{\iota_*}
& {\,\,\mathrm{K}_1(A)\,\,} \ar[d]_-{\lambda_{I1*}+\lambda_{II1*}} \ar[r]^-{}
& {\,\,0\,\,}\\
{\,\,0\,\,} \ar[r]^-{}
&{\,\,\mathrm{K}_0(B)\,\,} \ar[r]_-{\pi_*'}
& {\,\,\mathrm{K}_0(F_1') \,\,} \ar[r]_-{\alpha'-\beta'}
& {\,\,\mathrm{K}_1(SF_2') \,\,} \ar[r]_-{\iota_*'}
& {\,\,\mathrm{K}_1(B)\,\,} \ar[r]^-{}
& {\,\,0\,\,}.}
$$
Note that $\lambda_I+\lambda_{II}\in C(A,B)$.
The diagram
$$
\xymatrixcolsep{2pc}
\xymatrix{
{\,\,0\,\,} \ar[r]^-{}
&{\,\,\mathrm{K}_0(A)\,\,} \ar[d]_-{0} \ar[r]^-{\pi_*}
& {\,\,\mathrm{K}_0(F_1)\,\,} \ar[d]_-{0} \ar[r]^-{\alpha-\beta}
& {\,\,\mathrm{K}_1(SF_2)\,\,} \ar[d]_-{0} \ar[r]^-{\iota_*}
& {\,\,\mathrm{K}_1(A)\,\,} \ar[d]_-{0} \ar[r]^-{}
& {\,\,0\,\,}\\
{\,\,0\,\,} \ar[r]^-{}
&{\,\,\mathrm{K}_0(B)\,\,} \ar[r]_-{\pi_*'}
& {\,\,\mathrm{K}_0(F_1') \,\,} \ar[r]_-{\alpha'-\beta'}
& {\,\,\mathrm{K}_1(SF_2') \,\,} \ar[r]_-{\iota_*'}
& {\,\,\mathrm{K}_1(B)\,\,} \ar[r]^-{}
& {\,\,0\,\,},}
$$
to be denoted by 0, is the (unique) zero element of $C(A,B)$.
(Clearly, $\lambda+0=\lambda$ for $\lambda\in C(A,B)$.)

Given a commutative diagram $\lambda\in C(A,B)$,
$$
\xymatrixcolsep{2pc}
\xymatrix{
{\,\,0\,\,} \ar[r]^-{}
& {\,\,\mathrm{K}_0(A)\,\,} \ar[d]_-{\lambda_{0*}} \ar[r]^-{\pi_*}
& {\,\,\mathrm{K}_0(F_1)\,\,} \ar[d]_-{\lambda_{0}} \ar[r]^-{\alpha-\beta}
& {\,\,\mathrm{K}_1(SF_2)\,\,} \ar[d]_-{\lambda_{1}} \ar[r]^-{\iota_*}
& {\,\,\mathrm{K}_1(A)\,\,} \ar[d]_-{\lambda_{1*}} \ar[r]^-{}
& {\,\,0\,\,}\\
{\,\,0\,\,} \ar[r]^-{}
& {\,\,\mathrm{K}_0(B)\,\,} \ar[r]_-{\pi_*'}
& {\,\,\mathrm{K}_0(F_1') \,\,} \ar[r]_-{\alpha'-\beta'}
& {\,\,\mathrm{K}_1(SF_2') \,\,} \ar[r]_-{\iota_*'}
& {\,\,\mathrm{K}_1(B)\,\,} \ar[r]^-{}
& {\,\,0\,\,},}
$$
the inverse of $\lambda$, to be denoted by $-\lambda$, is
$$
\xymatrixcolsep{2pc}
\xymatrix{
{\,\,0\,\,} \ar[r]^-{}
& {\,\,\mathrm{K}_0(A)\,\,} \ar[d]_-{-\lambda_{0*}} \ar[r]^-{\pi_*}
& {\,\,\mathrm{K}_0(F_1)\,\,} \ar[d]_-{-\lambda_{0}} \ar[r]^-{\alpha-\beta}
& {\,\,\mathrm{K}_1(SF_2)\,\,} \ar[d]_-{-\lambda_{1}} \ar[r]^-{\iota_*}
& {\,\,\mathrm{K}_1(A)\,\,} \ar[d]_-{-\lambda_{1*}} \ar[r]^-{}
& {\,\,0\,\,}\\
{\,\,0\,\,} \ar[r]^-{}
& {\,\,\mathrm{K}_0(B)\,\,} \ar[r]_-{\pi_*'}
& {\,\,\mathrm{K}_0(F_1') \,\,} \ar[r]_-{\alpha'-\beta'}
& {\,\,\mathrm{K}_1(SF_2') \,\,} \ar[r]_-{\iota_*'}
& {\,\,\mathrm{K}_1(B)\,\,} \ar[r]^-{}
& {\,\,0\,\,}.}
$$
Note that $-\lambda\in C(A,B)$, and $\lambda+(-\lambda)=0$.

Then $C(A,B)$ is an Abelian group, and $M(A,B)$ is a subgroup of $C(A,B)$.
\end{notion}

\begin{theorem}[\cite{AE}]\label{CM}
Let $A$$,\,B \in \mathcal{C}$. Then we have a natural isomorphism of groups
$$
\mathrm{KK}(A,B)\cong C(A,B)/M(A,B).
$$
\end{theorem}
\begin{lemma}[Lemma 2.5 of \cite{AE}]\label{unitary}
Let $A,\,B \in \mathcal{C}$ be minimal. Let $\lambda$ be a commutative diagram,
$$
\xymatrixcolsep{2pc}
\xymatrix{
{\,\,0\,\,} \ar[r]^-{}
& {\,\,\mathrm{K}_0(A)\,\,} \ar[d]_-{\lambda_{0*}} \ar[r]^-{\pi_*}
& {\,\,\mathrm{K}_0(F_1)\,\,} \ar[d]_-{\lambda_{0}} \ar[r]^-{\alpha-\beta}
& {\,\,\mathrm{K}_1(SF_2)\,\,} \ar[d]_-{\lambda_{1}} \ar[r]^-{\iota_*}
& {\,\,\mathrm{K}_1(A)\,\,} \ar[d]_-{\lambda_{1*}} \ar[r]^-{}
& {\,\,0\,\,}\\
{\,\,0\,\,} \ar[r]^-{}
& {\,\,\mathrm{K}_0(B)\,\,} \ar[r]_-{\pi_*'}
& {\,\,\mathrm{K}_0(F_1') \,\,} \ar[r]_-{\alpha'-\beta'}
& {\,\,\mathrm{K}_1(SF_2') \,\,} \ar[r]_-{\iota_*'}
& {\,\,\mathrm{K}_1(B)\,\,} \ar[r]^-{}
& {\,\,0\,\,},}
$$
where the map $\lambda_0$ is positive. Let
$$\tau=\bigoplus_{j'=1}^{l'} \tau_{j'}:\,A\to C([0,1],F_2')$$
be a homomorphism such that $Sp(\pi_0'\circ\tau),\,Sp(\pi_1'\circ\tau)\subset Sp(F_1)$ and
$$
\mathrm{K}_0({\pi_0'}\circ\tau)=\alpha'\circ\lambda_0
\quad
and
\quad
\mathrm{K}_0({\pi_1'}\circ\tau)=\beta'\circ\lambda_0,
$$
where $\pi_0'$ and $\pi_1'$ are the point evaluations at 0 and 1, respectively.
Then there exists a unitary $u\in{C([0,1],F_2)}$ such that ${\rm Ad}\,u\circ\tau$ gives a homomorphism from $A$ to $B$.
\end{lemma}

\section{Perturbation and standard form}

\begin{definition}\label{defstd}
  If $A=A(F_1,F_2,\varphi_0,\varphi_1)\in\mathcal{C}$, we shall say that a homomorphism $\phi:A\rightarrow M_r(C[0,1])$ has standard form if:

(i) $\phi$ has the expression
$$
\phi_t(f)=
U(t)^*
\left(
\begin{array}{cccc}
  f(s_1(t)) &  &  &  \\
   & f(s_2(t)) &  &  \\
   &  & \ddots &  \\
   &  &  & f(s_k(t))
\end{array}
\right)U(t),\quad \forall\, f\in A, \quad (*)
$$
where $U\in M_r(C[0,1])$ and $\{s_i(t)\}_{i=1}^k\subset C([0,1],Sp(F_1))\cup C([0,1],Sp(SF_2))$.

(ii) Each of $\{s_i(t)\}_{i=1}^{k}$  has one of the following basic forms:
$$
\{\theta_1,\cdots,\theta_p,(t,1),\cdots,(t,l),(1-t,1),\cdots,(1-t,l)\},
$$
where $\theta_j$ is the $j^{th}$ point in $Sp(F_1)$ and $(t,m)$ denotes the point $t$ on the $m^{th}$ copy of the interval $(0,1)$ in $Sp(SF_2)$.
\end{definition}
\begin{definition}
  Let $n\in \mathbb{N}^{+}$, $A\in \mathcal{C}$ with $Sp(A)=Sp(F_1)\cup \coprod_{i=1}^{l}(0,1)_i$. Set $I_i^0=(0,\frac{1}{n}]_i,$$\cdots, I_i^r=[\frac{r-1}{n},\frac{r}{n}]_i,$$\cdots,I_i^{n-1}=[\frac{n-1}{n},1)_i$ for  $i=1,2,\cdots,l$. Then $Sp(A)=Sp(F_1)\cup \coprod_{i=1}^{l}\bigcup_{r=1}^{n}I_i^r$. We shall call this an $n$-partition of $Sp(A)$ and refer to $I_1^1,\cdots,I_i^r,\cdots,I_l^n$ as dividing intervals.
\end{definition}
\begin{definition}
 Let $A,B\in \mathcal{C}$ and let $\phi$ be a homomorphism from $A$ to $B$.  We shall say
 $\phi$ has $n$-standard form if after identifying each dividing interval of $Sp(B)$ with $[0,1]$, $\phi$ has standard form on each dividing interval.
\end{definition}

\begin{example}\label{piece}
Choose $A$ as in Example \ref{EX 1}. Define $\phi: A\rightarrow M_3(C[0,1])$ as follows:
$$
\phi(f,a\oplus b\oplus c)=
\begin{cases}
  \left(
  \begin{array}{cc}
    a &  \\
     & f(2t)
  \end{array}
  \right), & \mbox{if } 0\leq t \leq \frac{1}{2}; \\
  \left(
  \begin{array}{ccc}
    1 & 0 & 0 \\
    0 & 0 & 1 \\
    0 & 1 & 0
  \end{array}
  \right)
  \left(
  \begin{array}{cc}
    f(2t-1) &  \\
     & b
  \end{array}
  \right)
  \left(
    \begin{array}{ccc}
    1 & 0 & 0 \\
    0 & 0 & 1 \\
    0 & 1 & 0
  \end{array}
  \right)
  , & \mbox{if } \frac{1}{2}\leq t \leq 1.
\end{cases}
$$
Then $\phi$ has 2-standard form.
\end{example}

  The following result is essentially contained in Lemma 3.5 of \cite{Liu2}.
\begin{lemma}\label{appro3}
  Let $A\in \mathcal{C}$, $F\subset A$ be a finite subset, and $\delta>0$. There exist a finite subset $G$ and $\varepsilon>0$
  such that if $\psi,\varphi:A\rightarrow M_r(\mathbb{C})$ are homomorphisms, for some integer $r$, with
  $$
  \|\psi(g)-\varphi(g)\|<\varepsilon,\quad\forall\,g\in G,
  $$
  then there is a homomorphism $\phi:A\rightarrow M_r(C[0,1])$ such that $\phi_0=\psi$, $\phi_1=\varphi$, and
  $$
  \|\phi_t(f)-\psi(f)\|<\delta,\quad\forall\,f\in F,\,\,t\in [0,1].
  $$
  Moreover, $\phi$ may be chosen to be homotopic to a homomorphism with 3-standard form.
\end{lemma}
\begin{proof}
  The first part is included in \cite[Lemma 3.5]{Liu2} and the homomorphism $\phi$ has the expression $(*)$(condition (i) in  Definition \ref{defstd}) on $[0,\frac{1}{3}]$, $[\frac{1}{3},\frac{2}{3}]$ and $[\frac{2}{3},1]$. For the second part, consider the interval $[0,\frac{1}{3}]$, after identifying $[0,\frac{1}{3}]$ with $[0,1]$, for each $k$, we require that $s_k(t)$ satisfies one of the following conditions:

(1) $s_k(t)\in \{\theta_1,\cdots,\theta_p\}$;

(2) $Im s_k(t)\subset (0,1)_{i_k}$ for some $i_k\in\{1,2,\cdots,l\}$;

(3) $s_k(t)\in (0,1)_{i_k},(0\leq t<1)$ and $s_k(1)=0_{i_k}$ for some $i_k$;

(4) $s_k(t)\in (0,1)_{i_k},(0\leq t<1)$ and $s_k(1)=1_{i_k}$ for some $i_k$;

(5) $s_k(t)\in (0,1)_{i_k},(0<t\leq1)$ and $s_k(0)=0_{i_k}$ for some $i_k$;

(6) $s_k(t)\in (0,1)_{i_k},(0<t\leq1)$ and $s_k(0)=1_{i_k}$ for some $i_k$;

(7) $s_k(t)\in (0,1)_{i_k},(0< t<1)$ and $s_k(0)=s_k(1)=0_{i_k}$ for some $i_k$;

(8) $s_k(t)\in (0,1)_{i_k},(0<t<1)$ and $s_k(0)=s_k(1)=1_{i_k}$ for some $i_k$;

(9) $s_k(t)\in (0,1)_{i_k},(0<t<1)$ and $s_k(0)=0_{i_k},\ s_k(1)=1_{i_k}$ for some $i_k$;

(10) $s_k(t)\in (0,1)_{i_k},(0<t<1)$ and $s_k(0)=1_{i_k}$, $s_k(1)=0_{i_k}$ for some $i_k$.

Then we use the following basic techniques:

(a) If $s_k(t)$ satisfies (2), (3), (5) or (7), then $s_k(t)$ is homotopic to $\widetilde{s}_k(t)=0_{i_k}$;

(b) If $s_k(t)$ satisfies (4), (6) or (8), then $s_k(t)$ is homotopic to $\widetilde{s}_k(t)=1_{i_k}$;

(c) If $s_k(t)$ satisfies (9), then $s_k(t)$ is homotopic to $\widetilde{s}_k(t)=(t,i_k)$;

(d) If $s_k(t)$ satisfies (10), then $s_k(t)$ is homotopic to $\widetilde{s}_k(t)=(1-t,i_k)$.

Now we get a new homomorphism $\phi^{(1)}$ which has standard form on $[0,\frac{1}{3}]$. Note that after the homotopy step on the interval $[0,\frac{1}{3}]$, $\phi^{(1)}$ still has the expression $(*)$ on the interval $[\frac{1}{3},\frac{2}{3}]$. Then for each $k$, $s_k(t)$ still satisfys one of the ten conditions above, so we repeat the same techniques for $[\frac{1}{3},\frac{2}{3}]$ and $[\frac{2}{3},1]$, in this way, we obtain a homomorphism $\phi^{(3)}$ with 3-standard form. This shows that $\phi$ is homotopic to a homomorphism of $3$-standard form.

\end{proof}
 The following lemma is  \cite[Corollary 4.3]{B}.
\begin{lemma}\label{homolift}
Let $A$ be a semiprojective $\mathrm{C}^*$-algebra generated by a finite
or countable set $\mathcal{G}=\{x_1,x_2,\cdots\}$ with $\underset{j\rightarrow\infty}{\lim}\|x_j\|=0$ if $\mathcal{G}$ is infinite. Then
 there is a $\delta>0$ such that, whenever $B$ is a $\mathrm{C}^*$-algebra, and $\phi_0$ and
$\phi_1$ are homomorphisms from $A$ to $B$ with $\|\phi_0(x_j)-\phi_1(x_j) \|<\delta$ for all $j$, then $\phi_0$ and $\phi_1$ are homotopic. (The number $\delta$ depends on $A$ and the set $\mathcal{G}$ of generators, but not on $B$ or the maps $\phi_0,\phi_1$.)
\end{lemma}

\begin{theorem}\label{homo to m-standard}
Let $A,B\in \mathcal{C}$. Then any homomorphism from $A$ to $B$ is homotopic to a homomorphism of $m$-standard form for some $m$.
\end{theorem}
\begin{proof}
Note that $A$ is semiprojective and finitely generated by \cite{ELP1}. Set $\mathcal{G}$ be the finite set generates $A$, with $\delta>0$ as in Lemma \ref{homolift}, apply Lemma \ref{appro3} for $A$, $\mathcal{G}$, $\delta/2$ and $M_k(\mathbb{C})$, with $k$ large enough that $B\subset M_k(C[0,1])$, to obtain a finite subset $G\subset A$ and $\varepsilon>0$.

Let $\phi:A\rightarrow B$ be a homomorphism, by Lemma \ref{basichom}, we may assume that $Sp(\pi_e\circ\phi)\subset Sp(F_1)$.
Since $G\cup\mathcal{G}$ is finite, then there exist $n\in \mathbb{N}^+$ and an $n$-partition of $Sp(B)$, such that for any $x,y$ in the same dividing interval, we have
 $$
 \|\phi_x(g)-\phi_y(g)\|<\min\{\varepsilon,\frac{\delta}{2}\},\quad\forall\,g\in G\cup \mathcal{G}.
 $$

 Consider each dividing interval $I_i^r\subset Sp(B)$, and let $x=(\frac{r}{n},i),y=(\frac{r+1}{n},i)$. By Lemma \ref{appro3}, there exists a homomorphism $\psi|_{I_i^r}:A\rightarrow M_k(C(I_i^r))$ such that
 $$
 \psi_{(\frac{r}{n},i)}=\phi_{(\frac{r}{n},i)},\quad \psi_{(\frac{r+1}{n},i)}=\phi_{(\frac{r+1}{n},i)}
 $$ and
 $$
 \|\psi_t(g)-\phi_{(\frac{r}{n},i)}(g)\|<\frac{\delta}{2},\quad\forall\,g\in \mathcal{G},\,\,t\in I_i^r.
 $$
 The $\psi|_{I_i^r}$ fit together to define a single homomorphism $\psi:A\rightarrow B$ and $Sp(\pi_e\circ\psi)\subset Sp(F_1)$.

 It is clear that
 $$
 \|\phi|_{I_i^r}(g)-\psi|_{I_i^r}(g)\|<\delta,\quad\forall\,g\in \mathcal{G},
 $$
 then we have
 $$
 \|\phi(g)-\psi(g)\|<\delta,\quad\forall\,g\in \mathcal{G}.
 $$
 By the conclusion of Lemma \ref{homolift}, $\phi$ is homotopic to $\psi$.

  Consider the $3n$-partition of $Sp(B)$, $\psi$ has the expression $(*)$ on each dividing interval $J_1^1,\cdots,J_i^r,\cdots,J_l^{3n}$. Beginning with $J_1^1$, using the  basic homotopic techniques in Lemma \ref{appro3}, $\psi$ is homotopic to a homomorphism $\psi^{(1)}$ which has standard form on $J_1^1$, step by step, we can obtain that $\psi$ is homotopic to $\psi^{(3n)}$ and $\psi^{(3n)}$ has $3n$-standard form on all the dividing intervals. Then $\phi$ is homotopic to a homomorphism of $3n$-standard form.
\end{proof}

\section{Stable homotopy}
At the beginning of this section, we list one more example and write down some homotopy paths, the idea of which also works for Example \ref{EX 2}.
\begin{example}\label{EX 4}
Suppose that $A=C[0,1]$, $B=C(S^1)$, and $\psi$ is a 2-standard form homomorphism from $A$ to $B$:
$$
\psi(f)=\begin{cases}
  f(2t), & \mbox{if } 0\leq t\leq \frac{1}{2} \\
  f(2-2t), & \mbox{if } \frac{1}{2}\leq t\leq 1
\end{cases}.
$$
Denote by $\eta$ the 1-standard form homomorphism from $A$ to $B$ that
$$
\eta(f)=f(1).
$$
Note that the homomorphism $\psi\oplus \eta$ is homotopic to a 3-standard form homomorphism $\rho$,
$$
\rho(f)=\begin{cases}
\left(
\begin{array}{cc}
  f(3t) &  \\
   & f(1)
\end{array}
\right), & \mbox{if } 0\leq t\leq \frac{1}{3} \\
  \left(
\begin{array}{cc}
  f(1)  &  \\
   & f(1)
\end{array}
\right), & \mbox{if } \frac{1}{3}\leq t\leq \frac{2}{3} \\
   \left(
\begin{array}{cc}
  f(3-3t)  &  \\
   & f(1)
\end{array}
\right), & \mbox{if } \frac{2}{3}\leq t\leq 1
\end{cases},
$$
which can also be written as
$$
\rho(f)=\begin{cases}
v(t)\left(
\begin{array}{cc}
  f(3t) &  \\
   & f(1)
\end{array}
\right)v(t)^*, & \mbox{if } 0\leq t\leq \frac{1}{3} \\
 v(t) \left(
\begin{array}{cc}
  f(1)  &  \\
   & f(1)
\end{array}
\right)v(t)^*, & \mbox{if } \frac{1}{3}\leq t\leq \frac{2}{3} \\
  v(t) \left(
\begin{array}{cc}
  f(1)  &  \\
   & f(3-3t)
\end{array}
\right)v(t)^*, & \mbox{if } \frac{2}{3}\leq t\leq 1
\end{cases},
$$
where
$$
v(t)=\begin{cases}
\left(
\begin{array}{cc}
  1 &  \\
   & 1
\end{array}
\right), & \mbox{if } 0\leq t\leq \frac{1}{3} \\
  \left(
\begin{array}{cc}
  \cos (\frac{\pi}{2}\cdot(3t-1))  & \sin (\frac{\pi}{2}\cdot(3t-1)) \\
   -\sin (\frac{\pi}{2}\cdot(3t-1))& \cos (\frac{\pi}{2}\cdot(3t-1))
\end{array}
\right), & \mbox{if } \frac{1}{3}\leq t\leq \frac{2}{3} \\
   \left(
\begin{array}{cc}
    & 1 \\
   1&
\end{array}
\right), & \mbox{if } \frac{2}{3}\leq t\leq 1
\end{cases}.
$$
From this we can see that $\rho$ is homotopic to the 1-standard form homomorphism $\rho'$,
$$
\rho'(f)=
v(t)\left(
\begin{array}{cc}
  f(t) &  \\
   & f(1-t)
\end{array}
\right)v(t)^*
$$
\end{example}
The construction above can be generalized as follows:
\begin{remark}\label{trick}
Suppose that $A,B\in \mathcal{C}$, $\psi:\,A\to B$ is a homomorphism, and $\iota'$ is the natural embedding map from $B$ to $F_2'\otimes C[0,1]$.

For $t_0\in (0,1)$, denote by $\eta_{t_0}$ the homomorphism from $A$ to $F_2'\otimes C[0,1]$ defined by
$$
\eta_{t_0}(f)=\psi(f)(t_0)\in F_2'\otimes C[0,1],\quad \forall f\in A.
$$

Note that $\iota'\circ\psi\sim_h \rho$, where
$$
\rho(f)(t)=
\begin{cases}
\psi(f)(3t_0\cdot t), & \mbox{if } 0\leq t\leq \frac{1}{3} \\
  \psi(f)(t_0), & \mbox{if } \frac{1}{3}\leq t\leq \frac{2}{3} \\
  \psi(f)(3(1-t_0)t+3t_0-2), & \mbox{if } \frac{2}{3}\leq t\leq 1
\end{cases}.
$$
Also,
$$
\rho\oplus\eta_0= V(t)\rho' V(t)^*,
$$
where
$$
\rho'(f)(t)=\begin{cases}
\left(
\begin{array}{cc}
  \psi(f)(3t_0\cdot t) &  \\
   & \psi(f)(t_0)
\end{array}
\right) , & \mbox{if } 0\leq t\leq \frac{1}{3} \\
  \left(
\begin{array}{cc}
  \psi(f)(t_0) &  \\
   & \psi(f)(t_0)
\end{array}
\right) , & \mbox{if } \frac{1}{3}\leq t\leq \frac{2}{3} \\
   \left(
\begin{array}{cc}
  \psi(f)(t_0) &  \\
   & \psi(f)(3(1-t_0)t+3t_0-2)
\end{array}
\right), & \mbox{if } \frac{2}{3}\leq t\leq 1
\end{cases},
$$
$$V(t)=\begin{cases}
\left(
\begin{array}{cc}
  1 &  \\
   & 1
\end{array}
\right)\otimes 1_{F_2'\otimes C[0,1]}, & \mbox{if } 0\leq t\leq \frac{1}{3} \\
  \left(
\begin{array}{cc}
  \cos (\frac{\pi}{2}\cdot(3t-1))  & \sin (\frac{\pi}{2}\cdot(3t-1)) \\
   -\sin (\frac{\pi}{2}\cdot(3t-1))& \cos (\frac{\pi}{2}\cdot(3t-1))
\end{array}
\right)\otimes 1_{F_2'\otimes C[0,1]}, & \mbox{if } \frac{1}{3}\leq t\leq \frac{2}{3} \\
   \left(
\begin{array}{cc}
    & 1 \\
   1&
\end{array}
\right)\otimes 1_{F_2'\otimes C[0,1]}, & \mbox{if } \frac{2}{3}\leq t\leq 1
\end{cases}.
$$
We also have
$
\rho'\sim_h \rho'',
$
where
$$
\rho''(f)(t)=   \left(
\begin{array}{cc}
    \psi(f)(\frac{t}{t_0})&  \\
   &\psi(f)(\frac{t-t_0}{1-t_0})
\end{array}
\right).
$$
That is, we have
$$
\phi\oplus\eta_0\sim_h V(t)\rho''V(t)^*.
$$
Furthermore, the homotopy path $H_\theta$ with $\theta\in [0,1]$, $H_0=\phi\oplus\eta_0$ and  $H_1=V(t)\rho''V(t)^*$, can be chosen to satisfy
$$
\pi_0'\circ H_\theta(f)=\left(
\begin{array}{cc}
    \psi(f)(0)&  \\
   &\psi(f)(t_0)
\end{array}
\right) \quad {\rm and}\quad
\pi_1'\circ H_\theta(f)=\left(
\begin{array}{cc}
    \psi(f)(1)&  \\
   &\psi(f)(t_0)
\end{array}
\right),
$$
where $\pi_0',\,\pi_1'$ are the point evaluations of $M_2(F_2\otimes C[0,1])$ at 0 and 1.

As the homomorphism $\eta_{t_0}\otimes 1_B$ has the form
$$
u (\bigoplus^{r}\eta_{t_0})u^*,\quad {\rm for~ some~} u\in C([0,1],M_r(F_2')),
$$
where $r=\sum_{j'=1}^{l'}h_{j'}'$.
For homomorphism $\psi\oplus (\eta_{t_0}\otimes 1_B)=(1_B\oplus u)(\psi\oplus\bigoplus^{r}\eta_{t_0})(1_B\oplus u^*)$,
we can apply the same construction to $\psi\oplus\bigoplus^{r}\eta_{t_0}$
by regarding the first copy of $\eta_{t_0}$ in the diagonal as $\eta_{t_0}$,
we will have a homotopy path from $A$ to $M_r(B)$.
\end{remark}

The idea of the following lemma comes from Example \ref{EX 2}, Example \ref{EX 4}, and Remark \ref{trick}.

\begin{lemma}\label{m to 1}
Let $A,B\in \mathcal{C}$. Given any $m$-standard form homomorphism $\psi$ from $A$ to $B$, there exists a homomorphism $\eta$ from $A$ to $M_r(B)$, for some integer $r$, such that $\psi\oplus \eta$ is homotopic to a homomorphism with 1-standard form.
\end{lemma}
\begin{proof}

Consider the homomorphism
$$
\eta:=\bigoplus_{k=1}^{m-1}\psi_{\frac{k}{m}}\otimes 1_B,\quad A \to M_r(B),
$$
where $r=(m-1)\sum_{j'=1}^{l'}h_{j'}'$, and $\psi_{\frac{k}{m}}:\, A\to F_2'$ is the point evaluation of $B$ at $\frac{k}{m}$ composed with $\psi$.
Applying the trick we describe in Remark \ref{trick} $m-1$ times to the points  $\{\frac{k}{m}\}_{k=1}^{m-1}$,
we obtain that $\psi\oplus \eta$ is homotopic to a homomorphism of 1-standard form.

\end{proof}
{\bf Recall that for a $m$-standard form, unitary is just piecewise continuous (Example \ref{piece}). But when we consider the 1-standard form, the unitary is a continuous element in} $\mathbf{F_2'\otimes C[0,1]}.$

We list the basic homotopy lemma from functional calculus:
\begin{proposition}\label{homotopy lemma}
Let $D\subset M_n(\mathbb{C})$ be a sub $\mathrm{C}^*$-algebra. Let $v$ be a unitary in $M_n(\mathbb{C})$ with
$$
vdv^*=d, \quad {\rm for~all}~ d\in D.
$$
Then there exists a unitary path $V_t$, $t\in [0,1]$, in $M_n(\mathbb{C})$, with $V_0=v,\,V_1=1_n$,
inducing the relation
$$
v\sim_h 1_n,
$$
such that
$$
V_t dV_t^*=d,\quad  {\rm for~all}~  d\in D ~ {\rm and}~  t\in [0,1].
$$

\end{proposition}
Suggested by Example \ref{EX 3}, we have
\begin{lemma}\label{same diagram}
Let $A,B\in \mathcal{C}$, $\psi_1,\psi_2$ are two 1-standard form homomorphisms from $A$ to $B$,
inducing the same $\lambda\in C(A,B)$,
  $$
\xymatrixcolsep{2pc}
\xymatrix{
{\,\,0\,\,} \ar[r]^-{}
& {\,\,\mathrm{K}_0(A)\,\,} \ar[d]_-{\lambda_{0*}} \ar[r]^-{\pi_*}
& {\,\,\mathrm{K}_0(F_1)\,\,} \ar[d]_-{\lambda_{0}} \ar[r]^-{\alpha-\beta}
& {\,\,\mathrm{K}_1(SF_2)\,\,} \ar[d]_-{\lambda_{1}} \ar[r]^-{\iota_*}
& {\,\,\mathrm{K}_1(A)\,\,} \ar[d]_-{\lambda_{1*}} \ar[r]^-{}
& {\,\,0\,\,}\\
{\,\,0\,\,} \ar[r]^-{}
& {\,\,\mathrm{K}_0(B)\,\,} \ar[r]_-{\pi_*'}
& {\,\,\mathrm{K}_0(F_1') \,\,} \ar[r]_-{\alpha'-\beta'}
& {\,\,\mathrm{K}_1(SF_2') \,\,} \ar[r]_-{\iota_*'}
& {\,\,\mathrm{K}_1(B)\,\,} \ar[r]^-{}
& {\,\,0\,\,}.}
$$
Then there is a homomorphism $\zeta$ with finite dimensional image from $A$ to $M_r(B)$ such that
$$
 \psi_1\oplus \zeta\sim_h\psi_2\oplus \zeta.
 $$
\end{lemma}

\begin{proof}
As the 1-standard form homomorphisms $\psi_1,\psi_2$ induce the same diagram, we have that for each ${j'}^{\rm th}$ block of $B$,
the differences of the multiplicities of $f_j(t)$ and $f_j(1-t)$ in $\psi_1,\psi_2$ are the same, which is equal to the $(j',j)^{\rm th}$ entry of $\lambda_1$.
On one hand, from Example \ref{EX 4},
we know that
$$
\rho'\sim_h\psi\oplus \eta\sim_h \xi,
$$
where
$$
\xi(f)=\left(
\begin{array}{cc}
  f(0) &  \\
   & f(1)
\end{array}
\right).
$$
On the other hand, using the same trick as in Example \ref{EX 4} and Remark \ref{trick},
we have
$$
 \psi_1\sim_h \widetilde{\psi}_1=u(t)\chi  u(t)^*\quad {\rm for\,\,some}~u(t)\in U(F_2'\otimes C[0,1])
$$
and
$$
 \psi_2\sim_h \widetilde{\psi}_2=v(t) \chi  v(t)^*\quad {\rm for\,\,some}~ v(t)\in U(F_2'\otimes C[0,1]).
$$

(Here, we can require both $u(t)\chi  u(t)^*$ and $v(t)\chi  v(t)^*$ to have 1-standard form with $\chi$  satisfing that at each ${j'}^{\rm th}$ block of $B$, at least one of the multiplicities of $f_j(t)$ and $f_j(1-t)$ is 0.)

We hope that $u(t)v(t)^*$ is a unitary in $B$. This is not always true, but we claim that $u(t) v(t)^*$
is homotopic to a unitary in $B$.
This is because there exist $\zeta_1,\zeta_2: F_1\to F_1'$ such that we have the following commutative diagrams:
$$
\xymatrixcolsep{3pc}
\xymatrix{
{\,\,A\,\,} \ar[r]^-{u(t)\chi  u(t)^*} \ar[rd]_-{\pi}
& {\,\,B\,\,} \ar[r]^-{\pi'}
& {\,\,F_1'\,\,}\\
{\,\,\,\,}
& {\,\,F_1\,\,} \ar[ur]^-{\zeta_1}}\quad
{\rm and}\quad
\xymatrixcolsep{3pc}
\xymatrix{
{\,\,A\,\,} \ar[r]^-{v(t)\chi  v(t)^*} \ar[rd]_-{\pi}
& {\,\,B\,\,} \ar[r]^-{\pi'}
& {\,\,F_1'\,\,}\\
{\,\,\,\,}
& {\,\,F_1\,\,} \ar[ur]^-{\zeta_2}}.
$$
As $\zeta_1,\zeta_2$ induce the same map $\lambda_0$ between $\mathrm{K}_0(F_1)$ and $\mathrm{K}_0(F_1')$,
there exists a unitary $W\in F_1'$ such that
$$
W \zeta_2 W^*=\zeta_1,
$$
$$
\varphi_0'(W)(\pi_0 \circ\widetilde{\psi}_2)\varphi_0'(W)^*=u(0)v(0)^*(\pi_0 \circ\widetilde{\psi}_2)(u(0)v(0)^*)^*,\quad {\rm and}
$$
$$
\varphi_1'(W)(\pi_1 \circ\widetilde{\psi}_2)\varphi_1'(W)^*=u(1)v(1)^*(\pi_1 \circ\widetilde{\psi}_2)(u(1)v(1)^*)^*.
$$
Applying Proposition \ref{homotopy lemma}, we have
$$
\varphi_0'(W)^*u(0)v(0)^*\sim_h 1_{F_2'},\quad{\rm by}\,\,V_s,\,s\in[0,1],
$$
and
$$
\varphi_1'(W)^*u(1)v(1)^*\sim_h 1_{F_2'},\quad{\rm by}\,\,\widetilde{V}_s,\,s\in[0,1].
$$
Note that
$$
u(t)v^*(t)\sim_h w(t),
$$
where
$$
w(t)=\begin{cases}
\varphi_0'(W)\cdot V_{(1-3t)}, & \mbox{if } 0\leq t\leq \frac{1}{3} \\
  u(3t-1)v^*(3t-1), & \mbox{if } \frac{1}{3}\leq t\leq \frac{2}{3} \\
  \varphi_1'(W)\cdot V_{(3t-2)}, & \mbox{if } \frac{2}{3}\leq t\leq 1
\end{cases}\in U(B),
$$
and this gives
$$
\widetilde{\psi}_1=u(t)v(t)^*\widetilde{\psi}_2 v(t)u(t)^*\sim_h w(t)\widetilde{\psi}_2 w(t)^*.
$$
Using the trick from Example \ref{EX 3},
we deduce that
$$
(w(t)\widetilde{\psi}_2w(t)^*)\oplus (w(t)^* 0 w(t))\sim_h \widetilde{\psi}_2\oplus (w(t)^* 0 w(t)).
$$
\end{proof}

Example \ref{EX 1} shows that two homomorphisms with the same $\mathrm{KK}$-class may induce different diagrams in $C(A,B)$, and in that example we need to add enough point evaluations and push the spectrum point from 0 to 1. Suggested by this example, we will need:

\begin{lemma}\label{different diagram homotopic}
Let $\psi_1,\psi_2$ two 1-standard form homomorphisms from $A$ to $B$ inducing the same $\mathrm{KK}$-class. Then there is a homomorphism $\zeta$ with finite dimensional image from $A$ to $M_r(B)$, for some integer $r$ such that
$$
\psi_1\oplus \zeta\sim_h\psi_2\oplus \zeta.
$$
\end{lemma}

\begin{proof}
$\psi_1,\psi_2$ are two 1-standard form homomorphisms inducing $\lambda_1,\lambda_2\in C(A,B)$.
If $\lambda_1=\lambda_2$, then by Lemma \ref{same diagram}, the proof is finished.

As $\mathrm{KK}(\psi_1)=\mathrm{KK}(\psi_2)$, if $\lambda_1\neq\lambda_2$, by Theorem \ref{CM},
then there exists a map $\mu:\,\mathrm{K_1}(SF_2)\to \mathrm{K_0}(F_1')$ inducing an element $\lambda_\mu\in M(A,B)$ such that
$$
\lambda_1+\lambda_\mu=\lambda_2.
$$
So, as suggested by Lemma \ref{same diagram}, we wish to show that there exists a homomorphism $\eta$, with finite dimensional image and inducing $\xi \in C(A,B)$, such that
$\psi_1\oplus\eta$ is homotopic to a 1-standard form homomorphism inducing the diagram $\lambda_2+\xi$, which is also the diagram induced by $\psi_2\oplus\eta$. And we only need to show that this is true when $\mu$ is any matrix unit, since the stably homotopic relation is an equivalent relation.
($\mu$ is a finite sum of matrix units.)

We shall deal with the case $\mu=e_{11}$; the other cases are similar. Then, $\lambda_\mu\in M(A,B)$ is the diagram
 $$
\xymatrixcolsep{2pc}
\xymatrix{
{\,\,0\,\,} \ar[r]^-{}
& {\,\,\mathrm{K}_0(A)\,\,} \ar[d]_-{0} \ar[r]^-{\pi_*}
& {\,\,\mathrm{K}_0(F_1)\,\,} \ar[d]_-{e_{11}(\alpha-\beta)} \ar[r]^-{\alpha-\beta}
& {\,\,\mathrm{K}_1(SF_2)\,\,} \ar[d]_-{(\alpha'-\beta')e_{11}} \ar[r]^-{\iota_*}
& {\,\,\mathrm{K}_1(A)\,\,} \ar[d]_-{0} \ar[r]^-{}
& {\,\,0\,\,}\\
{\,\,0\,\,} \ar[r]^-{}
& {\,\,\mathrm{K}_0(B)\,\,} \ar[r]_-{\pi_*'}
& {\,\,\mathrm{K}_0(F_1') \,\,} \ar[r]_-{\alpha'-\beta'}
& {\,\,\mathrm{K}_1(SF_2') \,\,} \ar[r]_-{\iota_*'}
& {\,\,\mathrm{K}_1(B)\,\,} \ar[r]^-{}
& {\,\,0\,\,}.}
$$
Let $\kappa$ be the following diagram in $C(A,B)$:
$$
\xymatrixcolsep{2pc}
\xymatrix{
{\,\,0\,\,} \ar[r]^-{}
& {\,\,\mathrm{K}_0(A)\,\,} \ar[d]_-{\kappa_{0*}} \ar[r]^-{\pi_*}
& {\,\,\mathrm{K}_0(F_1)\,\,} \ar[d]_-{\kappa_{0}} \ar[r]^-{\alpha-\beta}
& {\,\,\mathrm{K}_1(SF_2)\,\,} \ar[d]_-{0} \ar[r]^-{\iota_*}
& {\,\,\mathrm{K}_1(A)\,\,} \ar[d]_-{0} \ar[r]^-{}
& {\,\,0\,\,}\\
{\,\,0\,\,} \ar[r]^-{}
& {\,\,\mathrm{K}_0(B)\,\,} \ar[r]_-{\pi_*'}
& {\,\,\mathrm{K}_0(F_1') \,\,} \ar[r]_-{\alpha'-\beta'}
& {\,\,\mathrm{K}_1(SF_2') \,\,} \ar[r]_-{\iota_*'}
& {\,\,\mathrm{K}_1(B)\,\,} \ar[r]^-{}
& {\,\,0\,\,},}
$$
where
$$
\kappa_{0}=
\begin{pmatrix}
  k_1' & k_1' & \cdots & k_1' \\
   k_2'& k_2' & \cdots & k_2' \\
   \vdots & \vdots & \ddots & \vdots \\
   k_p' & k_p' & \cdots & k_p'
 \end{pmatrix}_{p'\times p}.
$$
Recall that $(k_1',k_2',\cdots,k_p')=[1_B]$ is the scale in $\mathrm{K}_0(B)$.


There exists a large enough integer $c$ ($ck_1'\geq \|e_{11}\beta\|_\infty$, and the following $\Delta_{j'i}\geq 0$, for all $j'=1,2,\cdots,l'$, $i=1,2,\cdots,p$) that we have the following 1-standard form homomorphism $\eta_0$ inducing $c\kappa+\lambda_\mu$.

For any  $ j'\in\{1,2,\cdots,l'\}$,
define a homomorphism from $A$ to the algebra $M_r(M_{{h'}_{j'}}(C[0,1]))$ (for some integer $r$):
$$
A\ni (f,a)\stackrel{\phi_{j'}}{\longmapsto}g_{j'}\in M_r(M_{{h'}_{j'}}(C[0,1]))$$
with
\begin{align*}
g_{j'}(t)&= {\rm diag}\{\underbrace{f(t,1),f(t,1),\cdots,f(t,1)}_{\alpha_{j'1}}\}
\\
&\qquad\qquad \oplus
 {\rm diag}\{\underbrace{f(1-t,1),f(1-t,1),\cdots,f(1-t,1)}_{\beta_{j'1}}\}
\\
&\qquad\qquad \oplus
\bigoplus_{i=1} ^p {\rm diag}\{\underbrace{a(\theta_i),a(\theta_i),\cdots,a(\theta_i)}
_{\Delta_{j'i}}\},
\end{align*}
where
$$\Delta_{j'i}=
(\alpha'\circ(e_{11}(\alpha-\beta)+c\kappa_0))_{j'i}-
(\alpha_{11}\cdot\alpha_{j'1}'
+
\beta_{ji}\cdot\beta_{j'1}').$$
Then, with
$$\phi=\bigoplus_{j'=1}^{l'}\phi_{j{'}},$$
by Lemma~\ref{unitary} we have a unitary $u$
such that $\eta_0=u \phi u^*$ is a homomorphism from $A$ to $M_r(B)$ inducing the commutative diagram $c\kappa+\lambda_\mu$.

Furthermore, we can choose the above $u$ who satisfies that
$u H_s u^*$ is also a homomorphism for every $s\in [0,1]$ with
$$
Sp(\theta_{i'}\circ H_s)\cap \coprod_{j=1}^{l} (0,1)_j=\emptyset,\quad i'=2,3,\cdots,p',\,\,s\in(0,1),
$$
where $H_s$ is a homotopy path of homomorphisms from $A$ to $M_r(F_2\otimes C[0,1])$. Write
$$
H_s(f)=\bigoplus_{j'=1}^{l'}\phi_{s,j{'}},
$$
where
$$
A\ni (f,a)\stackrel{\phi_{s,j'}}{\longmapsto}g_{s,j'}\in M_r(M_{{h'}_{j'}}(C[0,1]))$$
with
\begin{align*}
g_{s,j'}(t)&= {\rm diag}\{\underbrace{f(st,1),f(st,1),\cdots,f(st,1)}_{\alpha_{j'1}}\}
\\
&\qquad\qquad \oplus
 {\rm diag}\{\underbrace{f(1-st,1),f(1-st,1),\cdots,f(1-st,1)}_{\beta_{j'1}}\}
\\
&\qquad\qquad \oplus
\bigoplus_{i=1} ^p {\rm diag}\{\underbrace{a(\theta_i),a(\theta_i),\cdots,a(\theta_i)}
_{\Delta_{j'i}}\}.
\end{align*}
Then $u H_s u^*$ induces a homotopy
$$
\eta_0\sim_h \eta,
$$
where $\eta=u H_1 u^*$ is a homomorphism inducing the diagram $c\kappa$ with finite dimensional image.
(One can imagine for comparison the simple example that the identity map from C[0,1] to C[0,1] is homotopic to the point evaluation at 1.)

Now we have
$$
\psi_1\oplus \eta\sim_h\psi_1\oplus \eta_0,
$$
and $\psi_1\oplus \eta_0$ induces $\lambda_1+c\kappa+\lambda_\mu=\lambda_2+c\kappa$.

For the general $\mu=\mu_{i'j}e_{i'j}$, we repeat the above construction $\sum_{j=1}^{l} \sum_{i'=1}^{p'}\mid\mu_{i'j}\mid$ times. The conclusion of the lemma follows with $\zeta$ is the sum of such $\sum_{j=1}^{l} \sum_{i'=1}^{p'}\mid\mu_{i'j}\mid$ homomorphisms with finite dimensional images.

\end{proof}
\begin{proof}[Proof of Theorem \ref{main result}]
We only need to show that two homomorphism $\psi_1,\psi_2$ inducing the same $\mathrm{KK}$-class are stably homotopic. By Theorem \ref{homo to m-standard}, we have the two homomorphisms $\psi_1,\psi_2$ homotopic to $\psi_1',\psi_2'$, which are of $m_1$- and $m_2$-standard form, respectively.
By Lemma \ref{m to 1}, there exists a homomorphism $\eta_1$ such that the homomorphisms
$$\psi_1'\oplus\eta_1,\psi_2'\oplus\eta_1$$ are homotopic to two 1-standard form homomorphisms $\psi_1'',\psi_2''$, respectively. And at last, by Lemma \ref{different diagram homotopic}, there exists a homomorphism $\eta_2$ such that
$$\psi_1''\oplus\eta_2\sim_h\psi_2''\oplus\eta_2.$$
That is,
$$
\psi_1\oplus\eta_1\oplus\eta_2\sim_h\psi_2\oplus\eta_1\oplus\eta_2,
$$
as desired.
\end{proof}

The property (H) of $\mathrm{C}^*$-algebras plays an important role in the classification theory, we use stable homotopy to prove that any one-dimensional NCCW complex has this property.

\begin{definition}[\cite{D1}]
A $\mathrm{C}^*$-algebra is said to have property (H) if, for any finite $F \subset A$ and $\varepsilon > 0$, there exist a homomorphism $\rho :A \rightarrow M_{r-1}(A)$  and a homomorphism $\mu:A \rightarrow M_r(A)$ with finite dimensional image such that
  $$
  \|f\oplus\rho(f)-\mu(f)\|<\varepsilon,\quad \forall f\in F.
  $$
\end{definition}
\begin{theorem}
  Every unital one-dimensional NCCW complex has property (H).
\end{theorem}
\begin{proof}
  By Theorem \ref{CM}, the map $id:A\rightarrow A$ induces the following commutative diagram, $\lambda$
 $$
\xymatrixcolsep{2pc}
\xymatrix{
{\,\,0\,\,} \ar[r]^-{}
& {\,\,\mathrm{K}_0(A)\,\,} \ar[d]_-{\lambda_{0*}} \ar[r]^-{\pi_*}
& {\,\,\mathrm{K}_0(F_1)\,\,} \ar[d]_-{\lambda_{0}} \ar[r]^-{\alpha-\beta}
& {\,\,\mathrm{K}_1(SF_2)\,\,} \ar[d]_-{\lambda_{1}} \ar[r]^-{\iota_*}
& {\,\,\mathrm{K}_1(A)\,\,} \ar[d]_-{\lambda_{1*}} \ar[r]^-{}
& {\,\,0\,\,}\\
{\,\,0\,\,} \ar[r]^-{}
& {\,\,\mathrm{K}_0(A)\,\,} \ar[r]_-{\pi_*}
& {\,\,\mathrm{K}_0(F_1) \,\,} \ar[r]_-{\alpha-\beta}
& {\,\,\mathrm{K}_1(SF_2) \,\,} \ar[r]_-{\iota_*}
& {\,\,\mathrm{K}_1(A)\,\,} \ar[r]^-{}
& {\,\,0\,\,},}
$$
 where $\lambda_0=I_{p\times p}$, $\lambda_1=I_{l\times l}$.

 Construct a commutative diagram $\lambda'\in C(A,A)$ as follows:
 $$
\xymatrixcolsep{2pc}
\xymatrix{
{\,\,0\,\,} \ar[r]^-{}
& {\,\,\mathrm{K}_0(A)\,\,} \ar[d]_-{\lambda_{0*}'} \ar[r]^-{\pi_*}
& {\,\,\mathrm{K}_0(F_1)\,\,} \ar[d]_-{\lambda_{0}'} \ar[r]^-{\alpha-\beta}
& {\,\,\mathrm{K}_1(SF_2)\,\,} \ar[d]_-{\lambda_{1}'} \ar[r]^-{\iota_*}
& {\,\,\mathrm{K}_1(A)\,\,} \ar[d]_-{\lambda_{1*}'} \ar[r]^-{}
& {\,\,0\,\,}\\
{\,\,0\,\,} \ar[r]^-{}
& {\,\,\mathrm{K}_0(A)\,\,} \ar[r]_-{\pi_*}
& {\,\,\mathrm{K}_0(F_1) \,\,} \ar[r]_-{\alpha-\beta}
& {\,\,\mathrm{K}_1(SF_2) \,\,} \ar[r]_-{\iota_*}
& {\,\,\mathrm{K}_1(A)\,\,} \ar[r]^-{}
& {\,\,0\,\,},}
$$
 with

 $$
 \lambda_0'=L
 \begin{pmatrix}
  k_1 & k_1 & \cdots & k_1 \\
   k_2& k_2 & \cdots & k_2 \\
   \vdots & \vdots & \ddots & \vdots \\
   k_p & k_p & \cdots & k_p
 \end{pmatrix}-\lambda_0  \quad  and  \quad \lambda_1'=-\lambda_1,
 $$
where $L\in \mathbb{N}$ is large enough. Consider the sum $\zeta=\lambda+\lambda'$:
 $$
\xymatrixcolsep{2pc}
\xymatrix{
{\,\,0\,\,} \ar[r]^-{}
& {\,\,\mathrm{K}_0(A)\,\,} \ar[d]_-{\lambda_{0*}+\lambda_{0*}'} \ar[r]^-{\pi_*}
& {\,\,\mathrm{K}_0(F_1)\,\,} \ar[d]_-{\lambda_{0}+\lambda_{0}'} \ar[r]^-{\alpha-\beta}
& {\,\,\mathrm{K}_1(SF_2)\,\,} \ar[d]_-{\lambda_{1}+\lambda_{1}'} \ar[r]^-{\iota_*}
& {\,\,\mathrm{K}_1(A)\,\,} \ar[d]_-{\lambda_{1*}+\lambda_{1*}'} \ar[r]^-{}
& {\,\,0\,\,}\\
{\,\,0\,\,} \ar[r]^-{}
& {\,\,\mathrm{K}_0(A)\,\,} \ar[r]_-{\pi_*}
& {\,\,\mathrm{K}_0(F_1) \,\,} \ar[r]_-{\alpha-\beta}
& {\,\,\mathrm{K}_1(SF_2) \,\,} \ar[r]_-{\iota_*}
& {\,\,\mathrm{K}_1(A)\,\,} \ar[r]^-{}
& {\,\,0\,\,}.}
$$

Then we can construct a homomorphism $\phi:A\rightarrow M_m(A)$ of 1-standard form inducing the diagram $\lambda'$ for some $m$ (just as in the construction in the proof of Lemma \ref{different diagram homotopic}; see also \cite[Lemma 3.6]{AE}) and a homomorphism $\psi$ from $A$ to $M_{m+1}(A)$ with finite dimensional image inducing the diagram $\zeta$. Then by Theorem \ref{CM},
 $$
 KK(id\oplus\phi)=KK(\psi).
 $$
 Hence by Theorem \ref{main result}, there exists a homomorphism $\eta:A\rightarrow M_n(A)$ with finite dimensional image such that
 $$
 id\oplus\phi\oplus\eta\,\sim_h\, \psi\oplus\eta.
 $$

 Set $s=m+n+1$, $\sigma=id\oplus\phi\oplus\eta$, $t=\sum_{j=1}^{l}h_j$, for any $x\in (0,1)$, denote
 $$
 \mu_x=\bigoplus_{j=1}^l{\rm diag}\{\underbrace{\sigma_{(x,j)},\sigma_{(x,j)},\cdots,\sigma_{(x,j)}}_t\},
 $$
 then there exists a unitary $u$ such that $u\cdot\mu_x\cdot u^*$ gives a homomorphism from $A$ to $M_{st}(A)$.

 Choose $x_0,x_1,x_2,\cdots$ be a dense set of $(0,1)$. Let $D$ be a $C^*$-algebra constructed as an inductive limit $D=\lim (D_i,\nu_{i,j})$, where $D_i=M_{(s+st)^i}(A)$, $i\geq 0$ and $\nu_{i,i+1}(f)=\sigma(f)\oplus \mu_{x_i}(f)$.($\sigma$ extends to $M_k(A)$ for any $k$, here we still denote it by $\sigma$). A similar construction can be found in
 \cite[Lemma 6.1]{DG}. The perturbation coming from Remark \ref{trick} (by changing the $\frac{1}{3},\,\frac{2}{3}$ into $\frac{1}{2}-\delta$, $\frac{1}{2}+\delta$, respectively, with small $\delta>0$) shows that the inductive limit has real rank zero. Since the connecting maps always induce the zero map on $\mathrm{K}_1$-group, we have $K_1(D)=0$, and it follows from Theorem 3.1 of \cite{ELP2} that the inclusion $\nu_{0,\infty}$ of $D_0=A$ into $D$ can be approximated arbitrarily well by homomorphisms with finite dimensional range.

 Then for any finite set $F\subset A$ and $\varepsilon>0$, there exist $i$ and a homomorphism $\mu: A\rightarrow D_i$ with finite dimensional range such that $\|\nu_{0,i}(f)-\mu(f)\|<\varepsilon$ for all $f\in F$. Since $\nu_{0,i}$ is of the form $id\oplus \rho$ for some homomorphism $\rho$, this completes the proof.

\end{proof}

\section*{Acknowledgements}
The research of Qingnan An and Zhichao Liu was supported by the University of Toronto and NNSF of China (No.:11531003), both the the authors thank the  Fields Institute for their hospitality; the research of the second author was supported by the Natural Sciences and Engineering Council of Canada; the research of Yuanhang Zhang was partly supported by the Natural Science Foundation for Young Scientists of Jilin Province (No.:20190103028JH) and NNSF of China (No.:11601104, 11671167, 11201171) and the China Scholarship Council.

\end{document}